\documentclass[usenames,dvipsnames,11pt]{article}

\usepackage[utf8]{inputenc}
\usepackage[margin=1.3in]{geometry}
\usepackage[titletoc,title]{appendix}
\usepackage{hyperref}
\usepackage{amsmath}
\usepackage{fancyhdr}
\usepackage{amssymb}
\usepackage{amsthm}
\usepackage{sidecap}
\usepackage{booktabs}
\usepackage{xcolor}
\usepackage{graphicx}
\usepackage{float}
\usepackage{mathtools}
\usepackage{indentfirst}
\usepackage{mathrsfs}
\usepackage{titlesec}
\usepackage{listings}
\usepackage{amsbsy}
\usepackage{amsfonts}
\usepackage{latexsym}
\usepackage{amsopn}
\usepackage{amscd}
\usepackage{amsxtra}
\usepackage{euscript}
\usepackage{ dsfont }
\usepackage{ esint }
  \usepackage{mathrsfs}
 \usepackage{latexsym}
 \usepackage{bbm}
 \usepackage{hyperref}
 \usepackage{mathtools}
 \usepackage[font=small,format=hang,labelfont={sf,bf}]{caption}
 \usepackage{epsfig}
 \usepackage{subfig}
 \usepackage{url}
 \usepackage{varioref}
 \usepackage{bm}
 \usepackage{authblk}
\numberwithin{equation}{section}

\theoremstyle{plain}
\begingroup
\newtheorem{theorem}{Theorem}[section]

\newtheorem*{theorem*}{Theorem}
\newtheorem{lemma}[theorem]{Lemma}
\newtheorem{proposition}[theorem]{Proposition}
\newtheorem{corollary}[theorem]{Corollary}
\endgroup
\theoremstyle{definition}
\theoremstyle{remark}
\begingroup

\newtheorem{remark}[theorem]{Remark}

\endgroup
\newcommand{\bd}{\partial}

\newcommand{\R}{\mathbb{R}}
\newcommand{\Q}{\mathbb{Q}}
\newcommand{\N}{\mathbb{N}}

\newcommand{\ud}{\,\textnormal{d}}
\newcommand{\dist}{\text{dist}}
\newcommand{\virg}[1]{``#1''}

\newcommand{\udH}{\ud\mathcal H^{N-1}}
\let\dist\undefined
\newcommand{\dist}{\textnormal{dist}}
\newcommand{\sd}{\textnormal{sd}}
\newcommand{\Ls}{\mathcal{L}_s}

\newcommand{\h}{^{(h)}}

\usepackage[maxbibnames=99]{biblatex} 
\title{Asymptotic of the Discrete Volume-Preserving Fractional Mean Curvature Flow\\
via a Nonlocal Quantitative Alexandrov Theorem}

\author[1]{Daniele De Gennaro}
\author[2]{Andrea Kubin}
\author[3]{Anna Kubin}
\affil[1]{Dauphine University, Paris France}
\affil[2]{Technische Universität München, Munich Germany}
\affil[3]{Politecnico di Torino, Turin, Italy}
\date{}
\begin{document}
\maketitle

\begin{abstract}
\noindent We characterize the long time behaviour of a discrete-in-time approximation of the volume preserving fractional mean curvature flow. In particular, we prove that the discrete flow starting from any bounded set of finite fractional perimeter converges exponentially fast to a single ball.  As an intermediate result we establish a quantitative Alexandrov type estimate for normal deformations of a ball. Finally, we provide existence for flat flows as limit points of the discrete flow when the time discretization parameter tends to zero.
\end{abstract}

\tableofcontents

\section*{Introduction}
We consider the geometric evolution of sets called \textit{the volume preserving fractional mean curvature flow}. The \textit{classical mean curvature flow} is defined as a flow of sets $(E_t)_{0\leq t\leq T}$ in $\R^N$ following the motion law
\begin{equation*}
	v_t=-H_{E_t}\quad \text{on} \quad \bd E_t,
\end{equation*}
where $v_t$ denotes the component of the velocity relative to the outer normal vector of $\bd E_t$  and $H_E$ is the mean curvature of the set $E$.
In order to include the volume constraint, one can consider the following velocity
\begin{equation*}
	v_t= \bar{H}_{E_t}- H_{E_t}\quad \text{on} \quad \bd E_t
\end{equation*}
for all $t\in[0,T]$ , where $\bar{H}_{E_t}$ denotes the average of $H_{E_t}$ over $\bd E_t$. 
The defined geometric evolution is called  \textit{volume preserving mean curvature flow}, as one can observe that the volume of the evolving sets is constant.

In the fractional setting, the velocity of the flow is given by the \textit{fractional mean curvature}, a geometric quantity introduced by Caffarelli, Roquejoffre and Savin in \cite{CRS} and defined as the first variation of the fractional perimeter functional. The latter functional is defined on a measurable set $E\subset \R^N$ simply as
\[ P^s(E)=\int_E\int_{E^c} \dfrac1{|x-y|^{N+s}}\ud x\ud y. \]
It turns out that its first variation on any $C^2$ set $E$ is given by the formula 
\[ H^s_E(x):=\int_{\R^N}\dfrac{\chi_E(y)-\chi_{E^c}(y)}{|x-y|^{N+s}}\ud y\qquad \forall x\in\bd E. \]
In both the previous formulae, the integrals are intended in the principal value sense. In analogy with the classical case, the evolution law for the \textit{volume preserving fractional mean curvature flow} is given by 
\begin{equation}
	v_t=\bar H^s_{E_t}-H^s_{E_t}\quad \text{on} \quad \bd E_t,
	\label{evolution law}
\end{equation}
with the notations previously introduced.

Up to now, a satisfactory study of this type of evolution is still missing. While the evolution without the volume constraint is well-understood (see e.g. \cite{CMP2015,Imb}), the lack of a comparison principle in our case makes the study much harder. Moreover, the generated flow may present singularities of different kinds, as happens for the classical mean curvature flow: see \cite{CSV2018} some some explicit examples of pinch-like singularities.  In \cite{JLaM} a short-time existence is provided for the smooth flow \eqref{evolution law}, while existence of the smooth flow starting from convex sets (under suitable assumptions) is proved in \cite{CSV}.

We will then follow the approach of the celebrated papers by Luckhaus and Sturzenhecker \cite{LS} and Almgren, Taylor and Wang \cite{ATW} consisting in building a discrete-in-time approximation of the flow via a variational approach and then sending the time discretization parameter to zero. Our approach follows closely the work done by Mugnai, Seis and Spadaro in \cite{MSS}, where the variational problem studied incorporates a volume penalization to take into account the volume constraint.  First of all we define a discrete-in-time approximation of the flow that will be called the \textit{discrete flow}. Given any initial set $E_0$, with $|E_0|=m$, and a time-step $h>0$ we define $E_0\h:=E_0$ and, iteratively, for $n> 0$
\[ E_{n+1}\h\in \text{argmin}\left\lbrace P^s(F)+\dfrac 1h \int_{F} \sd_{ E_n\h}(x)\ud x + \dfrac 1{h^{\frac s{1+s}}}||F|-m|: F \subset \R^N \textnormal{ measurable} \right\rbrace,  \]
where $\sd_{ E_n\h}$ is the signed distance function from the set $ E_n\h$. We can define for every $t \ge 0$, the discrete flow by $E\h(t):=E_{[t/h]}\h$. We will prove that such a flow is well-defined. Any $L^1_{loc}$-limit point of this flow as the time-step $h$ converges to zero will be called a \textit{flat flow}. 
For the classical mean curvature flow, under the hypothesis of convergence of the perimeters, this approach produces global-in-time distributional solutions of the evolution law \ref{evolution law}, as shown in \cite{MSS}. In the fractional case, we fall short of this result. Nonetheless, we will prove in the Appendix the existence of flat flows, defined simply as $L^1_{loc}-$limit points of the discrete flow as $h\to0$, and address some of its continuity properties. Moreover, under the additional hypothesis of boundedness of the flow, we will prove that the flat flow is volume-preserving.

In the recent years, the study of the long time behaviour of the volume preserving mean curvature flow has attracted more and more attention. In the classical case, after some preliminary studies \cite{ES, Hui}, in a recent paper \cite{MPS}  the authors proved the asymptotic behaviour of the classical discrete flow by showing its convergence to unions of equal balls. Then, they improved their results  in \cite{JMPS}, proving uniform estimates with respects to the time parameter $h$  in dimension $N=2$, thus obtaining the same result for the classical flat flow. Also the situation in the periodic setting is quite studied, with results for the discrete flow in \cite{DeKu} and for the flat flow in dimensions $N=3,4$ in \cite{Nii}. In the fractional setting some recent results have been proved. For example, in \cite{CSV} the authors prove that the smooth flow starting from a convex set converges to a ball, up to translations possibly depending on time and under the hypothesis of equiboundedness for the fractional curvatures along the flow.

In this paper the long-time convergence analysis is developed in the fractional setting. The main result of the paper is Theorem \ref{main result} below. It provides a complete characterization of the long-time behaviour of the discrete fractional mean curvature flow starting from any bounded set of finite fractional perimeter, providing also an estimate on the convergence speed. We will assume that the dimension $N$ is such that any $\Lambda-$minimizer of the fractional perimeter is a smooth set. Namely, we will assume that either:
\begin{itemize}
    \item $N=2$;
    \item $N\le 7$ and $s\in(s_0,1)$, where $s_0$ is the constant of Proposition \ref{prop density estimates}, item \textit{ii)}.
\end{itemize}
This is a technical hypothesis that could be dropped if we knew that the evolving sets were smooth. In particular, it is essential to characterize the possible long-time limit points for the discrete flow. We are then able to prove the following result.
 
\begin{theorem}\label{main result}
Let $m$, $M>0$ and let $E_0$ be an initial bounded set with $P^s(E_0)\le M$, $|E_0|=m$. Then, for $h=h(s,M,m)>0$ small enough the following holds: for any discrete flow $E_n\h$ starting from $E_0$, there exists $\xi\in\R^N$ such that 
\[ E_n\h-\xi\to B^{(m)}\quad \text{in} \quad C^k \]
for all $k\in\N.$ Moreover, the convergence is exponentially fast.
\end{theorem}
We stress the difference between our result and the one holding in the classic setting, where the limit points of the discrete flow are in general unions of disjointed balls having the same radius. This is a peculiar feature of the nonlocal perimeter considered, that penalizes non-connected components. 

A crucial intermediate  result consists in generalizing the Alexandrov-type estimate \cite[Theorem~1.3]{MPS} and \cite[Theorem 1.3]{DeKu} (see also \cite{KM}) to the fractional setting. This result provides a stability inequality for $C^1-$normal deformations of balls.
We briefly give some definitions to present some further details.  Set $B=B_1(0)$ and let $f:\bd B\to \R$ be a function with $\|f\|_{L^{\infty}(\bd B)}$ is sufficiently small. The\textit{ normal deformation} $B_f$ of the set $B$ is defined as
\[ \bd B_f:=\{ x(1+f(x)) \ :\ x\in\bd B  \}. \]
A normal deformation $B_f$ is said to be of class $C^k$ if $f\in C^k(\bd E)$. The quantitative Alexandrov type estimate proved in \cite{MPS} is the following.
\begin{theorem}[Theorem 1.1 in \cite{MPS}]\label{Alex old}
There exist $\delta \in (0,1/2)$ and $C>0$ with the following property: for any $f \in C^1(\bd B) \cap H^2(\bd B)$ such that $\|f\|_{C^1(\bd B)} \le \delta$, $|B_f|=\omega_N$ and $\textnormal{bar}(B_f)= \int_{B_f} x \ud x=0$, we have
\[\|f\|_{H^1(\bd B)} \le C\|H_{B_f}-\bar H_{B_f}\|_{L^2(\bd B)} .\]
\end{theorem}
We are able to extend the previous result to the fractional setting. Namely, we obtain the following.
\begin{theorem}
\label{Alex}
    There exist $\delta >0$ with the following property: for any $f \in C^1(\bd B)$ such that $\|f\|_{C^1(\bd B)} \le \delta$, $|B_f|=\omega_N$ and $\textnormal{bar}(B_f)= \int_{B_f} x \ud x=0$,
    then
    \begin{itemize}
        \item[i)] for any $s\in (0,1)$, there exists $C=C(N,s)>0$ such that 
        \begin{equation*}
        \|f\|_{H^{\frac{1+s}{2}}(\bd B)^2} \le C\|H^s_{B_f}-\bar H^s_{B_f}\|_{L^2(\bd B)^2} ;
    \end{equation*}
    \item[ii)] there exist $s^*\in(0,1)$ and $C=C(N)>0$ such that, for any $s\in (s^*,1)$, it holds
        \begin{equation}\label{Alex eq}
        (1-s)\|f\|^2_{H^{\frac{1+s}{2}}(\bd B)} \le C \big \|(1-s)\left(H^s_{B_f}-\bar H^s_{B_f}\right) \big \|_{L^2(\bd B)}^2 ;
    \end{equation}
    \item[iii)] if $f\in C^2(\bd B)$, as $s\to 1,$ estimate \eqref{Alex eq} tends to 
    \begin{equation*}
        \|f\|_{H^1(\bd B)^2} \le C \| H_{B_f}-\bar H_{B_f}  \|_{L^2(\bd B)^2}.
    \end{equation*}
    \end{itemize}
    In the previous results, we have set  $\bar H^s_{B_f}:= \fint_{\bd B} H^s_{B_f}(x+f(x)x) \udH(x).$
\end{theorem}
In particular, we recover Theorem \ref{Alex old} as a corollary of our result. The proof of the previous theorem follows closely the proof of the quantitative Alexandrov type estimate obtained in the flat torus and contained in \cite{DeKu}. In particular, the approach is based on some Taylor approximations of the factor $\bar H^s_{B_f}-H^s_{B_f}(x)$ combined with the coercivity of the second variation of the fractional perimeter, proved in \cite{FFMMM}.

After this work was completed, we were informed that a similar quantitative Alexandrov type estimate, more precisely a result analogous to item \textit{i)} in Theorem \ref{Alex}, has been independently proved in \cite{CesNov22}. In this paper the authors use this result to prove both global existence of a flow of smooth sets satisfying \eqref{evolution law}, starting from a smooth normal deformation of a ball, and its asymptotic exponential convergence to the same ball.

\subsection*{Acknowledgements}

The authors warmly thank Professor  Massimiliano Morini for insightful discussions.
D.~De Gennaro and Anna Kubin wish to acknowledge the hospitality of the Faculty of Mathematics of the Technical University of Munich, where part of this research was carried out. Andrea Kubin and Anna Kubin are members of the Gruppo Nazionale per l’Analisi Matematica, la Probabilità e le loro Applicazioni (GNAMPA) of the Istituto Nazionale di Alta Matematica (INdAM).  
D.~De Gennaro has received funding from the European Union’s Horizon 2020 research and innovation programme under the Marie Skłodowska-Curie grant agreement No 94532. \includegraphics[scale=0.01]{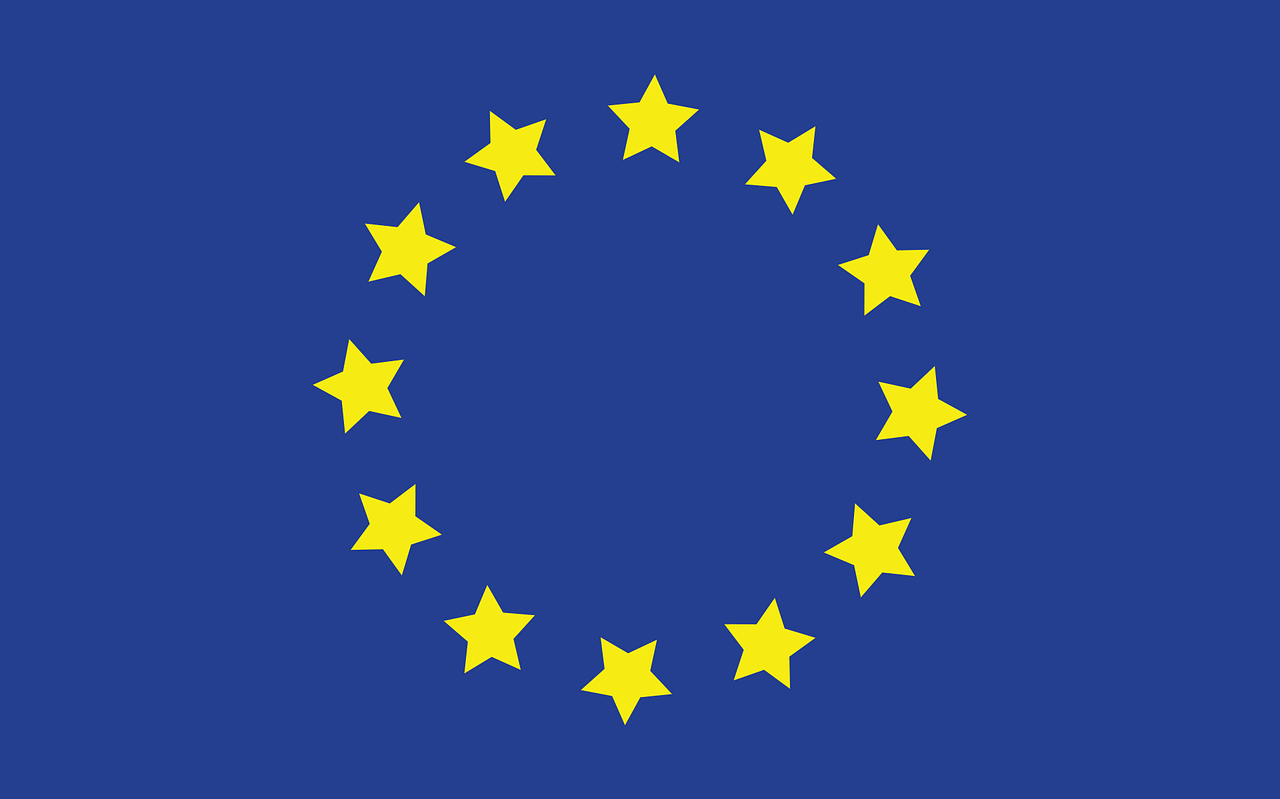}.
Andrea Kubin was supported by the DFG Collaborative Research Center TRR 109 “Discretization in Geometry and Dynamics”. 
	

\section*{Notation}
We work in the Euclidian space $\mathbb{R}^N$, with $N\geq 1$. We denote with $| \cdot |$ the standard Lebesgue measure in $\mathbb{R}^N$, $\mathrm{M}(\mathbb{R}^N)$ is the family of measurable set of $\mathbb{R}^N$ and $\mathrm{M}_f(\mathbb{R}^N) \subset \mathrm{M}(\mathbb{R}^N)$ is the family of subsets of $\mathbb{R}^N$ having finite measure. We denote with $E^c$ the complement of a set $E \subset \mathbb{R}^N$. We denote by $\mathcal{H}^{N-1}$ the Hausdorff measure, and sometimes denote $\udH_x:=\udH(x)$. If $E$ is a set with smooth boundary the outer normal to $E$ at a point $x$ in $\partial E$ is denoted by $\nu=\nu_{E}(x)$. We will denote $B=B(0,1)$. We will denote the ball of radius $r$ and center $x$ both as $B(x,r)$ and $B_r(x).$ Also, with $B^{(m)}$ we denote the ball centered in 0 and having volume $|B^{(m)}|=m.$ Finally, we denote by $C(*,\cdots,*)$ a constant that depends on $*, \cdots, *$; such a
constant may change from line to line.

\section{Preliminaries}

Let $s \in (0,1)$ we define the $s$-fractional perimeter as the following function 
$$ P^s: \mathrm{M}_f(\mathbb{R}^N) \rightarrow [0,+\infty], \quad P^s(E)\coloneqq \int_{E} \int_{E^c} \frac{1}{\vert x-y \vert^{N+s}}\ud x \ud y=\frac 12 [\chi_E]_{H^{\frac{s}{2}}}. $$
More in general, for every $E, F \in \mathrm{M}_f(\mathbb{R}^N)$ we set $$\Ls(E,F) \coloneqq \int_E \int_F \frac{1}{|x-y|^{N+s}} \ud x \ud y $$
and, for any bounded set $\Omega$, we define the fractional perimeter of $E$ relative to $\Omega$ as
$$ P^s(E;\Omega)\coloneqq \Ls( E \cap \Omega,E^c \cap \Omega) + \Ls( E \cap \Omega, E^c \setminus \Omega) + \Ls(E \setminus \Omega, E^c \cap \Omega ) .$$
Let $E \in \mathrm{M}_f(\mathbb{R}^N)$ be a set of class $C^2$. Given a vector field $X \in C^{1}_c(\R^N; \R^N)$, let
$$\Phi:\R \times \R^N \to \R^N, \quad \Phi(t,x)=x+t X(x).$$
We recall that the first variation of the $s$-fractional perimeter of $E$ in the direction of $X$ is given by 
$$\delta P^s(E)[X] \coloneqq \frac{\ud}{\ud t} \Big |_{t=0} P^s(\Phi(t,E))= \int_{\bd E} H^s_{ E}(x) X(x) \cdot \nu_E(x) \udH_x, $$
where $H^s_E(x)$ is the $s$-fractional mean curvature of $E$ evaluated at $x \in \bd E$, that is
$$H^s_E(x) \coloneqq \int_{\R^N} \frac{\chi_E(y)-\chi_{E^c}(y)}{|x-y|^{N+s}} \ud y, $$
where the integral has to be intended in the principal value sense. Applying the divergence theorem, the fractional curvature can be written as
\[ H^s_E(p)=\int_{\bd E} \dfrac{(x-p)\cdot \nu_E(x)}{|x-p|^{N+s}} \udH(x) \quad \forall \, p \in \partial E.\]

We recall some useful results concerning sets of finite fractional perimeter.
\begin{proposition}[Lower semi-continuity]
	Let $\{E_n\}_{n \in \mathbb{N}}\subset \mathrm{M}_{f}(\mathbb{R}^N)$ such that $ \chi_{E_n} \rightarrow \chi_{E}$ in $L^1_{loc}$, as $n \rightarrow + \infty$, for some $E \in \mathrm{M}_f(\mathbb{R}^N)$. Then, for all $s \in (0,1)$, we have 
	$$ P^s(E)\leq \liminf_{n \rightarrow + \infty} P^s(E_n).$$
\end{proposition}

\begin{theorem}[Compactness]
	If $R>0$ and $\{ E_n\}_{n \in \mathbb{N}} \subset \mathrm{M}(\mathbb{R}^N)$, with 
	$$ E_n \subset B(0,R) \quad \forall n \in \mathbb{N} \quad \textnormal{and} \quad \sup_{n \in \mathbb{N}} P^s(E_n) < + \infty,$$
then, up to a subsequence, ${E_n} \rightarrow E$ in  $L^1(\mathbb{R}^N)$, where $E \subset B(0,R)$ and $P^s(E)< +\infty$.
\end{theorem}
 \begin{theorem}(Relative isoperimetric inequality)
    Let $\Omega \subset \mathbb{R}^N$ be an open bounded set with Lipschitz boundary and let $E\subset \R^N$ be a measurable set. Then there exists a constant $C=C(s,N, \Omega)>0$ such that
$$ P^{s}(E, \Omega) \geq \mathcal{L}_s(E \cap \Omega, E^c \cap \Omega) \geq C\min \left\{ \vert E \cap \Omega
 \vert^{ \frac{N-s}{N} }, \, \vert E \setminus \Omega \vert^{ \frac{N-s}{N} }\right\}. $$
 \end{theorem}

We recall the following pointwise convergence theorems. The first one concerns the convergence of the fractional perimeter to the classical one: its proof can be found in \cite[Theorem 1]{CV2013}.

\begin{theorem}\label{convergence Ps}
    Let $E$ be a bounded, $C^{1,\alpha}$ set for $\alpha\in(0,1)$. Then,
    \[ \lim_{s\to 1}(1-s)P^s(E)=\omega_{N-1}P(E). \]
\end{theorem}

The second one relates to the pointwise convergence of the fractional curvatures. It was proved in a more general setting in \cite{AV,CV2013,CDLNM}.
\begin{theorem}\label{convergence Hs}
    Let $E$ be a bounded, $C^{2}$ set. Then,
    \[ \lim_{s\to 1}(1-s)H^s_E=\omega_{N-1}H_E \]
    uniformly on $\bd E$.
\end{theorem}

Finally, we recall the pointwise convergence of the fractional Gagliardo seminorms to the Sobolev one. The classical proof is contained in \cite[Corollary 2]{BBM}, see also \cite[Proposition 3.7]{HP} for the same result in a more general setting.

\begin{theorem}\label{convergence seminorms}
Assume $ f \in H^s(\bd B)$. Then
$$ \lim_{s \to 1} (1-s) [f]_{H^{\frac{1+s}2}(\bd B)}^2= C\| \nabla f \|_{L^2(\bd B)}^2,$$
where $C>0 $ is a constant that depends only on $N$.
\end{theorem}

\section{A fractional  quantitative Alexandrov type estimate}
In this section, we are going to prove the quantitative Alexandrov inequality Theorem~\ref{Alex} in the nonlocal setting of the fractional perimeter. From now on we set
\[[f]_{\frac{1+s}{2}}^2\coloneqq [f]_{H^{\frac{1+s}{2}}(\bd B)}^2=\int_{\bd B}\int_{\bd B} \frac{|f(x)-f(y)|^2}{|x-y|^{N+s}} \udH_x \udH_y.  \]
We start by recalling representation formulas for the $s$-fractional perimeter and its first variation on smooth sets.
\begin{lemma}
The following equalities hold true:
\begin{enumerate}
    \item If $f \in C^1(\bd B)$ with $\|f\|_{\infty}$ sufficiently small, then
    \begin{equation}\label{P^s5}
    \begin{split}
        P^s(B_f)&=\frac{P^s(B)}{P(B)} \int_{\bd B} (1+f)^{N-s} \udH+\\
        &\quad+ \frac{1}{2} \int_{\bd B} \int_{\bd B} \int_{1+f(y)}^{1+f(x)} \int_{1+f(y)}^{1+f(x)} F_{|x-y|}(r,\rho) \ud r \ud \rho \udH_x \udH_y,
    \end{split}
    \end{equation}
    where, for every $x,$ $y \in \bd B$, we have set
    \begin{equation*}
        F_{|x-y|}(r,\rho):= 
        \frac{r^{N-1} \rho^{N-1}}{|r x-\rho y|^{N+s}}, \quad \forall r,\rho \in (0,+\infty) .
    \end{equation*}
    
    \item If $f \in C^1(\bd B)$ with $\|f\|_{\infty}$ sufficiently small, then we have
    \begin{equation}\label{first var}
    \begin{split}
        &\delta P^s(E_f)[\psi]= (N-s)\frac{P^s(B)}{P(B)} \int_{\bd B}         (1+f)^{N-s-1} \psi \udH\\
        &+\int_{\bd B} \int_{\bd B} \int_{f(y)}^{f(x)} \left ( \psi(x)F_{|x-y|}(1+f(x),1+\rho)-\psi(y)F_{|x-y|}(1+f(y),1+\rho) \right ) .
    \end{split}
    \end{equation}
\end{enumerate}
\end{lemma}

\begin{proof}
By explicit computations one can obtain equation \eqref{P^s5}, see for example the calculations in the proof of \cite[Theorem 2.1]{FFMMM}.
To prove \eqref{first var}, we take the derivative
\begin{equation*}
    \dfrac{\ud}{\ud t}\Big|_{t=0}   P^s(B_{f+t \psi})
\end{equation*}
in formula \eqref{P^s5} and, recalling that
\begin{equation*}
\begin{split}
    \frac{\ud}{\ud t} \left [ \int_{\alpha(t)}^{\beta(t)} \int_{\alpha(t)}^{\beta(t)} f(r,\rho) \ud \rho \ud r\right ]&= \int_{\alpha(t)}^{\beta(t)} (f( \beta(t),\rho) \beta'(t)-f(\alpha(t),\rho) \alpha'(t) )\ud \rho \\
    &\quad+ \int_{\alpha(t)}^{\beta(t)} (f(r,\beta(t)) \beta'(t)-f(r,\alpha(t)) \alpha'(t)) \ud r
\end{split}
\end{equation*}
for every function $\alpha,$ $\beta:\R \to \R$ of class $C^1$ and $f \in L^1_{loc}(\R \times \R)$,
we conclude
\begin{align*}
\delta P^s(B_f)[\psi]= &\int_{\bd B} \int_{\bd B} \int_{1+f(y)}^{1+f(x)} \left ( \psi(x)F_{|x-y|}(1+f(x),\rho)-\psi(y)F_{|x-y|}(1+f(y),\rho) \right )  \ud \rho\\
&+(N-s)\frac{P^s(B)}{P(B)} \int_{\bd B} (1+f)^{N-s-1} \psi \udH.
\end{align*}
A simple change of coordinates then yields the thesis.
\end{proof}

\begin{lemma}
If $f \in C^1(\bd B)$ with $\|f\|_{\infty}$ sufficiently small, then we have
 \begin{equation}
    \begin{split}\label{varpsi1}
        \delta P^s(B_f)[1]&=(N-s)\frac{P^s(B)}{P(B)} \int_{\bd B} \left(1+(N-s-1)f+ O(f^2)) \right) \udH +O([f]_{\frac{1+s}{2}}^2  ),
    \end{split}
    \end{equation}
     \begin{equation}
    \begin{split}\label{varpsif}
        \delta P^s(B_f)[f]&=(N-s)\frac{P^s(B)}{P(B)} \int_{\bd B} \left(1+(N-s-1)f+ O(f^2)) \right) f\udH \\ \quad
        &+\int_{\bd B} \int_{\bd B}  \frac{(f(x)-f(y))^2 }{|x-y|^{N+s}}  +O ([f]_{\frac{1+s}{2}}^2) (\|f\|_{\infty}+\|f\|^2_{C^1}).
    \end{split}
    \end{equation}
\end{lemma}

\begin{proof}
First, we remark that, by expanding the first term in \eqref{first var}, we obtain
\begin{equation*}
    \begin{split}
      \delta P^s(B_f)[\psi]&= (N-s)\frac{P^s(B)}{P(B)} \int_{\bd B}         (1+(N-s-1)f+O(f^2)) \psi \udH \\
        &+\int_{\bd B} \int_{\bd B} \int_{f(y)}^{f(x)} \left ( \psi(x)F_{|x-y|}(1+f(x),1+\rho)-\psi(y)F_{|x-y|}(1+f(y),1+\rho) \right ) .
    \end{split}
    \end{equation*}
Now, set $u=f(x)$ and $v=f(y)$, and define the auxiliary function
$$g(1+u,1+v,\rho):= \psi(x)F_{|x-y|}(1+u,1+\rho)-\psi(y)F_{|x-y|}(1+v,1+\rho).$$
We remark that 
\begin{equation}
\label{here}
\begin{split}
    F_{|x-y|}(1+u,1+\rho)&= \frac{(1+(N-1)u+O(u^2))(1+(N-1) \rho +O(\rho^2))}{|(1+u) x-(1+\rho) y|^{N+s}}\\
    &= \frac{(1+(N-1)u+O(u^2))(1+(N-1) \rho +O(\rho^2))}{|(u-\rho)^2+ (1+u)(1+\rho) |x-y|^2|^{\frac{N+s}{2}}}\\
    &= \frac{1+(N-1)u+(N-1)\rho +O(u^2+\rho^2)}{|x-y|^{N+s}((u-\rho)^2/|x-y|^2 +u+\rho +u \rho +1)^{\frac{N+s}{2}}}.
    \end{split}
\end{equation}
Now, since $\|f\|_{C^1} \le \delta$ and $\rho $ is between $f(y)$ and $f(x)$, $|u-\rho|\le \delta |x-y|$, we can then expand \eqref{here} and obtain
\begin{equation*}
\begin{split}
     &\frac{1+(N-1)u+(N-1)\rho +O(u^2+\rho^2)}{|x-y|^{N+s}} \left (1-\frac{N+s}{2}\left(\frac{(u-\rho)^2}{|x-y|^2} +u+\rho +u \rho\right)\right)\\
     &=\frac{1+(N-1-\frac{N+s}{2})(u+\rho) }{|x-y|^{N+s}}-C\frac{(u-\rho)^2}{|x-y|^{N+s+2}}+ \frac{O(u^2+\rho^2)}{|x-y|^{N+s}},
    \end{split}
\end{equation*}
from which we deduce
\begin{align*}
    g(1+u,1+v,\rho)&= \frac{\psi(x)-\psi(y)}{|x-y|^{N+s}} +C_1\left (\dfrac{\psi(x)(u+\rho)}{|x-y|^{N+s}}
-\dfrac{\psi(y)(v+\rho)}{|x-y|^{N+s}}\right )\\
&+C_2\left (-\frac{\psi(x)(u-\rho)^2}{|x-y|^{N+s+2}}+\frac{\psi(y)(v-\rho)^2}{|x-y|^{N+s+2}}\right )+\frac{O(u^2+v^2+\rho^2)}{|x-y|^{N+s}},
\end{align*}
where the constants $C_{1,2}$ can also be non-positive.
Finally, we obtain
\begin{equation}
    \begin{split}
        \int_{\bd B} &\int_{\bd B} \int_{f(y)}^{f(x)} g(1+f(x),1+f(y),1+\rho) \\
        &\quad = 
    \int_{\bd B} \int_{\bd B}  \frac{(\psi(x)-\psi(y))(f(x)-f(y)) }{|x-y|^{N+s}} \\
    &\quad +C_1\int_{\bd B} \int_{\bd B}  \frac{(f(x)-f(y))(3f(x)/2 +f(y))}{|x-y|^{N+s}}\psi(x) \nonumber\\
    &\quad +C_2 \int_{\bd B} \int_{\bd B}  \frac{(\psi(x)-\psi(y))(f(x)-f(y))^3}{|x-y|^{N+s+2}}\\
    &\quad +O\left ( \int_{\bd B} \int_{\bd B}  \frac{(f(x)-f(y)) }{|x-y|^{N+s}}f^2 \right ).
    \end{split}
\end{equation}
If $\psi =1$, using the fact that, by symmetry,
\begin{equation}\label{simmetr}
    \int_{\bd B} \int_{\bd B}  \frac{(f(x)-f(y))f(x)}{|x-y|^{N+s}} =\dfrac 12[f]_{\frac{1+s}{2}}^2,
\end{equation}
we obtain
\begin{equation*}
\begin{split}
    &\int_{\bd B} \int_{\bd B} \int_{f(y)}^{f(x)} g(1+f(x),1+f(y),1+\rho) =O  ( [f]_{\frac{1+s}{2}}^2  ).
    \end{split}
\end{equation*}
If $\psi =f$, using again \eqref{simmetr}, we get
\begin{equation*}
 \begin{split}
        \int_{\bd B} \int_{\bd B} \int_{f(y)}^{f(x)} g(1+f(x),1+f(y),1+\rho) &= 
    \int_{\bd B} \int_{\bd B}  \frac{(f(x)-f(y))^2 }{|x-y|^{N+s}} \\
    &\quad +O ([f]_{\frac{1+s}{2}}^2) (\|f\|_{\infty}+\|f\|^2_{C^1}).
    \end{split}
\end{equation*}
\end{proof}

In order to prove Theorem \ref{Alex}, we need the following lemma, which states the coercivity of the second variation of the fractional perimeter of a ball with respect to normal deformations. Its proof is contained in \cite[Theorem 8.1]{FFMMM}. We start by defining 
\begin{equation}\label{first eig}
    \lambda_1^s:= s(N-s)\dfrac{P^s(B)}{P(B)}.
\end{equation}

\begin{lemma}\label{lemma 1.4}
There exists $\delta>0$ small  such that, if $f\in C^1(\bd B)$ with $\|f\|_{C^1(\bd B)}\le \delta$, $|B_f|=\omega_N$ and $\textnormal{bar}(B_f)=0$, then we have
\begin{align*}
    \delta^2 P^s(B)[f]&=\int_{\bd B}\int_{\bd B} \frac{(f(x)-f(y))^2}{|x-y|^{N+s}} \udH_x \udH_y-\lambda_1^s \int_{\bd B} |f|^2 \udH\\
    &\geq \dfrac 14\Big ([f]_{\frac{1+s}{2}}^2+\lambda_1^s\|f\|_{L^2(\bd B)}^2 \Big).
\end{align*}
\end{lemma}

We are now in position to prove Theorem \ref{Alex}.
\begin{proof}[Proof of Theorem \ref{Alex}]
We start by proving item \textit{i)}. Without loss of generality, we assume that $\|H^s_{B_f} - \bar H^s_{B_f}\|_{L^2}\leq1$.
Let $\Phi :\bd B \to \bd B_f \subset \R^N$ be the map defined by $\Phi (x)= (1+~f(x))x,$ by direct computations one can prove that 
$$J \Phi(x) = (1+f(x))^{N-1}(1+(1+f(x))^{-2}|\nabla f(x)|^{1/2}.$$ 
For every $\psi \in C^{1}(\bd B)$, let 
\begin{equation*}
      \overline{\psi}:\R^N \to \R^N, \quad \bar\psi(x):= \frac{x}{\vert x\vert} \psi\left(\frac{x}{\vert x\vert}\right).
\end{equation*} 
Employing the area formula we get
\begin{equation*}
\begin{split}
    \delta P^s(B_f)[\psi]&=  \int_{\bd B_f} H^s_{ B_f}\nu_{B_f} \cdot \overline{\psi}   \udH \\
    &=  \int_{\bd B} H^s_{ B_f}(p) \nu_{B_f}(p) \cdot x \psi(x) J\Phi(x) \udH_x \\
    &=\int_{\partial B} H^s_{ B_f}(p)\,\psi(x)\,(1+f(x))^{N-1}  \ud \mathcal{H}^{N-1}_x,
\end{split}
\end{equation*}
where we have set $p=(1+f(x))x$ (for more details see \cite[Section 1]{MPS} and \cite[Section 3]{DeKu}).
Now, by a simple Taylor expansion we obtain
\begin{equation}\label{var1}
    \delta P^s(B_f)[\psi] =\int_{\partial B} H^s_{ B_f}(p)\,\psi(x)\,(1+(N-1)f(x)+O(f^2))  \ud \mathcal{H}^{N-1}_x.
\end{equation}
We recall that 
\begin{equation*}
    H^s_{ B}(x)=(N-s) \frac{P^s(B)}{P(B)} \quad \textnormal{for all } x \in \bd B.
\end{equation*}
If $\psi =1$, by combining formulas \eqref{var1} and \eqref{varpsi1}, we infer
\begin{equation}\label{eq_lin3_gen}
\begin{split}
    \int_{\bd B}  (H^s_{ B_f}(p) -H^s_{ B} )(1+(N-1)f(x)+O(f^2))  \udH_x  = \int_{\bd B} O(f) \udH+ O([f]_{\frac{1+s}{2}}^2)
    \end{split}
\end{equation}
and if $\psi=f$, by combining equations \eqref{var1} and \eqref{varpsif}, we get
\begin{equation}\label{eq_lin6666_gen}
\begin{split}
    &\int_{\bd B}\int_{\bd B} \frac{(f(x)-f(y))^2}{|x-y|^{N+s}} \udH_x \udH_y - s(N-s) \frac{P^s(B)}{P(B)} \int_{\bd B} f^2 \udH\\
    & =
    \int_{\bd B} \left (H^s_{ B_f}(p) -H^s_{ B} \right)(1+(N-1)f(x)+O(f^2)) f(x) \udH_x \\
    &\quad  +O ([f]_{\frac{1+s}{2}}^2) (\|f\|_{\infty}+\|f\|^2_{C^1}).
    \end{split}
\end{equation}
Using the same arguments of the proof of \cite[Theorem 1.3]{DeKu} (see also \cite[Theorem 1.3]{MPS}) we can conclude. Anyway for the interested reader we report a sketch of the proof.

By \eqref{eq_lin3_gen}, for $\delta$ sufficiently small, using Hölder's inequality we obtain 
\begin{align*}
    \left| \bar H^s_{B_f}-H^s_B \right|
    &\le \left|-\fint_{\bd B}  (H^s_{B_f}-H^s_B)((N-1)f+O(f^2))\udH \right|\\
    &\quad +\int_{\bd B}O(|f|)\ud \mathcal H^{N-1}+O([f]^2_{\frac{1+s}{2}})\\
    &\le \left| \fint_{\bd B}  (H^s_{B_f}-\bar H^s_{B_f})((N-1) f+O(f^2))\udH \right| \\
    &\quad+ \left| \fint_{\bd B}  (\bar H^s_{B_f}-H^s_B) ((N-1) f+O(f^2))\udH \right| \\
    &\quad +\int_{\bd B}O(|f|)\ud \mathcal H^{N-1}+O([f]^2_{\frac{1+s}{2}})\\
    &\le \delta\, \dfrac{N-1+C\delta}{P(B)^{1/2}}\|H^s_{B_f}-\bar H^s_{B_f}\|_{L^2}+\delta  \left(N-1+C\delta \right )|\bar H^s_{B_f}-H^s_B|\\
    &\quad+ \int_{\bd B}O(|f|)\ud \mathcal H^{N-1}+O([f]^2_{\frac{1+s}{2}}),
\end{align*}
with $C=C(N)$. 
For $\delta $ small enough, recalling that $\|H_{B_f} - \bar H_{B_f}\|_{L^2}\leq1$, the previous inequality implies
\begin{equation}
    \frac 12 |\bar H^s_{B_f}-H^s_B|\le  C\delta \|H^s_{B_f}-\bar H^s_{B_f}\|_{L^2}+ \int_{\bd B}O(|f|)\ud \mathcal H^{N-1} +O([f]^2_{\frac{1+s}{2}})\le C\delta.
    \label{eq_lin7_gen}
\end{equation}
By \eqref{eq_lin6666_gen}, using again Hölder's inequality and by the previous remark, we get
\begin{align}
    \label{end_chain_gen}
    &\int_{\bd B}\int_{\bd B} \frac{(f(x)-f(y))^2}{|x-y|^{N+s}} \udH_x \udH_y - s(N-s) \frac{P^s(B)}{P(B)} \int_{\bd B} f^2\udH \nonumber\\
    &=\int_{\bd B}\left(  H^s_{B_f}(p)-H^s_B \right)( 1+(N-1) f + O(f^2))f\udH\nonumber\\
    &\quad+O ([f]_{\frac{1+s}{2}}^2) (\|f\|_{\infty}+\|f\|^2_{C^1})\nonumber\\
    &= \int_{\bd B} (H^s_{B_f}(p)-\bar H^s_{B_f})( 1+(N-1) f + O(f^2))f\udH\nonumber\\
    &\quad+\int_{\bd B}(\bar H^s_{B_f}-H^s_B)( 1+(N-1) f + O(f^2))f\udH\nonumber\\
    &\quad+O ([f]_{\frac{1+s}{2}}^2) (\|f\|_{\infty}+\|f\|^2_{C^1})\nonumber\\
    &\leq C\|H^s_{B_f}-\bar H^s_{B_f}\|_{L^2}\|f\|_{L^2}+ |\bar H^s_{B_f}-H^s_B|\int_{\bd B} (1+(N-1) f + O(f^2))f\udH\nonumber\\
    &\quad+O ([f]_{\frac{1+s}{2}}^2) (\|f\|_{\infty}+\|f\|^2_{C^1}).
\end{align}
Since $|B_f|=\omega_N$, we have \begin{equation}\label{nonzero mean curv}
    \Big | \int_{\bd B} f \udH \Big | = \int_{\bd B} O(f^2) \udH.
\end{equation}
By \eqref{nonzero mean curv} and \eqref{eq_lin7_gen}, we obtain
\[|\bar H^s_{B_f}-H^s_B|\int_{\bd B} (f+O(f^2))\udH  \le  \delta \int_{\bd B}O(f^2).\]
Finally, by the above inequality, \eqref{nonzero mean curv} again and by combining \eqref{end_chain_gen} with \eqref{eq_lin7_gen} 
we deduce that, for any $\eta>0$, it holds 
\begin{align}
   &\int_{\bd B}\int_{\bd B} \frac{(f(x)-f(y))^2}{|x-y|^{N+s}} \udH_x \udH_y - s(N-s) \frac{P^s(B)}{P(B)} \int_{\bd B} f^2\udH \nonumber\\
   &\quad \leq C\|H^s_{B_f}-\bar H^s_{B_f}\|_{L^2}\|f\|_{L^2}+C\delta (\|f\|_{L^2}^2+[f]_{\frac{1+s}{2}}^2)\label{half_gen}\\
  &\quad \leq \dfrac 1{\eta} C^2\|H^s_{B_f}-\bar H^s_{B_f}\|_{L^2}^2+\eta\|f\|_{L^2}^2+C\delta (\|f\|_{L^2}^2+[f]_{\frac{1+s}{2}}^2).
  \label{end_gen}
\end{align}
The conclusion then follows combining \eqref{end_gen} with Lemma \ref{lemma 1.4} and taking $\delta$ and $\eta$ sufficiently small. 

Item \textit{ii)} then follows from item \textit{i)}, remarking that \eqref{half_gen} can be read as 
\begin{align*}
   &(1-s)\left(\int_{\bd B}\int_{\bd B} \frac{(f(x)-f(y))^2}{|x-y|^{N+s}} \udH_x \udH_y - s(N-s) \frac{P^s(B)}{P(B)} \int_{\bd B} f^2\udH \right)\\
   &\quad \leq C\Big\|(1-s)\left( H^s_{B_f}-\bar H^s_{B_f}\right)\Big\|_{L^2}\|f\|_{L^2}+C\delta (\|f\|_{L^2}^2+[f]_{\frac{1+s}{2}}^2).
\end{align*}
Thus, we obtain 
\[ \dfrac{1-s}4\left( \lambda_1^s \|f\|^2_{L^2} + [f]^2_{\frac{1+s}2}  \right)\le \dfrac{C^2}\eta \Big\|(1-s)\left( H^s_{B_f}-\bar H^s_{B_f}\right)\Big\|^2_{L^2} + \eta \|f\|^2_{L^2} + C\delta (\|f\|_{L^2}^2+[f]_{\frac{1+s}{2}}^2).
\]
Recalling \eqref{first eig} and the pointwise convergence Theorem \ref{convergence Ps}, we conclude. Finally, item \textit{iii)} is a corollary of item \textit{ii)} and Theorems \ref{convergence Hs} and \ref{convergence seminorms}, recalling that Theorem \ref{convergence Ps} implies
\[ \lim_{s\to 1}(1-s)\lambda_1^s=(N-1)\omega_{N-1}. \]
We just remark that, since the convergence is uniform, $(1-s)\bar H^s_{B_f}\to \bar H_{B_f}.$
\end{proof}

\section{The asymptotic of the discrete volume-preserving mean curvature flow}
In this section we start by introducing the incremental minimum problem which defines the discrete-in-time approximation of the volume preserving fractional mean curvature flow.

Let $E \ne \emptyset$ be a bounded, measurable subset of $\R^N$. In the following we will always assume that $E$ coincides with its Lebesgue representative.
Fixed $h>0$, $m>0$, we consider the minimum problem
\begin{equation}
    \min \left \{ P^s(F)+\frac{1}{h} \int_{F} \sd_E(x) \ud x +\frac{1}{h^{\frac{s}{s+1}}} ||F|-m| : F \subset \R^N\right \},
    \label{pb_min_1}
\end{equation}
where $\sd_E(x):= \dist_E (x)-\dist_{E^c}(x)$ is the signed distance from the set $E$. Observe that the minimum problem \eqref{pb_min_1} is equivalent to the problem 
\begin{equation*}
    \min \left \{ P^s(F)+\frac{1}{h} \int_{F \triangle E} \dist_{\bd E}(x) \ud x +\frac{1}{h^{\frac{s}{s+1}}} ||F|-m| : F \subset \R^N \right \}.
\end{equation*}
We set $\mathcal{F}_h(\cdot,\, E): \mathrm{M}_f(\mathbb{R}^N) \rightarrow (-\infty,\,+\infty]$ the functional
$$ \mathcal{F}_h(F,\, E)=  P^s(F)+\frac{1}{h} \int_{F} \sd_E(x) \ud x +\frac{1}{h^{\frac{s}{s+1}}} ||F|-m|.$$

The following proposition recalls some properties of minimizers of problem \eqref{pb_min_1}.
\begin{proposition}\label{prop density estimates}
    Let $M>0, h>0, s\in(0,1)$ and $m>0$. Let  $E \subset \R^N$ be a bounded, measurable set such that $P^s(E) \le M$ and $|E|\le M$.
   Then, there exists a minimizer $F$ of \eqref{pb_min_1}, which is bounded. 
    Moreover, the following properties hold:
    \begin{itemize}
    \item[i)] There exists $\Lambda=\Lambda(h,N,s)>0$ such that $F$ is a $\Lambda$-minimizer of the perimeter, namely
    \begin{equation*}
        P(F) \le P(F') + \Lambda \lvert F \triangle F' \rvert
    \end{equation*}
    for all measurable set $F' \subset \R^N$ such that $\textnormal{diam}(F\triangle F')\le 1$.
    
    \item[ii)]
    The boundary $\bd F$ is of class $C^{2,\alpha}$ for any $\alpha \in (0,s)$ outside of a closed set $\Sigma$ of Hausdorff dimension at most $N-3$. Moreover, there exists $s_0 \in(0,1)$ such that, if $s \in (s_0,1)$, then $\bd F$ is of class $C^{1,\alpha}$ for any $\alpha \in (0,1)$ outside a closed set $\Sigma$ of Hausdorff dimension at most $N-8$.

    \item[iii)]  There exist $c_0=c_0(N,s)>0$ and a radius $r_0=r_0(h,N,s)>0$ such that for every $x \in \bd F \setminus \Sigma$ and $r \in (0,r_0]$ we have
    \begin{equation*}
        \lvert B_r(x) \cap F \rvert \ge c_0 r^N \quad \text{and} \quad \lvert B_r(x) \setminus F \rvert \ge c_0 r^N.
    \end{equation*}
    
    \item[iv)] The following Euler-Lagrange equation holds: there exists $\lambda \in \R$ such that for all $X \in C^1_c(\R^N,\R^N)$ we have
    \begin{equation}
        \int_{\bd F}\dfrac {\sd_{ E}}h X\cdot \nu_F \ud \mathcal{H}^{N-1}+\int_{\bd  F} H^s_{F} X \cdot \nu_F \ud \mathcal{H}^{N-1}=\lambda\int_{\bd F} X\cdot \nu_F  \ud \mathcal{H}^{N-1},
    \label{euler-lagrange equation}
    \end{equation}
    where $\lambda=\textnormal{sgn}(m-|F|) h^{-\frac s{1+s}}$ if $|F|\neq m$, otherwise it is the Lagrange multiplier associated with the volume penalization.
    
    \item[v)] There exist $k_0=k_0(h,N,s,M,m)\in \N$ and $d_0=d_0(h,N,s,M,m)>0$ such that $F$ is made up of at most $k_0$ connected components having diameter larger than $d_0$.
    
\end{itemize}
    
\end{proposition}

\begin{proof}
For the existence of minimizers of \eqref{pb_min_1} see for example \cite[Theorem~1.1]{CeNo}. The $\Lambda-$minimality property is easily deduced, for instance we can choose $\Lambda=2(h^{-1} +h^{-\frac s{1+s}})$. Concerning property $ii)$, it follows from \cite[Theorem 1.1]{CeNo} and \cite[Theorem 5]{CV2013}. The density esimates can be found in \cite[Theorem 4.1]{CRS}.
The Euler-Lagrange equation in item \textit{iv)} is easily deduced performing variations of \eqref{pb_min_1}. The bound on the number of connected components and on the diameter of the components follows from a covering argument, as in \cite[Proposition 2.3]{MPS}.
\end{proof}

By induction we can now define the \emph{discrete-in-time, volume preserving fractional mean curvature flow} $\{E_n\h\}_{n\in \N}$ and we will refer to it as the \textit{discrete flow.}
Let $E_0 \subset \R^N$ be a measurable set such that $\lvert E_0 \rvert = m$, we define $E_1\h$ as a solution of \eqref{pb_min_1} with $E_0$ instead of $E$. Assume that $E_k\h$ is defined for $1 \le k \le n-1$, we define $E_{n}\h$ as a solution of \eqref{pb_min_1} with $E$ replaced by $E_{n-1}\h$.  We  then recall the classical density estimate (see \cite[Theorem~4.1]{CRS}).

\begin{proposition}\label{density_prop}
There exists a constant $C=C(N,s)>0$ with the following property:
given $E \subset \R^N$, $R$, $\mu>0$ and $x_0 \in \bd E$ such that
\begin{equation*}
    P^s(E) \le P^s(E \setminus B_r(x_0))+\mu |E \cap B_r(x_0)| \quad \forall 0<r<R,
\end{equation*}
then 
\begin{equation*}
    C r^N \le |E \cap B_r(x_0)| \quad \forall 0<r<\min \{R, C_{s,N}\mu^{-1/s} \}.
\end{equation*}
\end{proposition}
\begin{corollary}
    Let $E \subset \R^N$ be a bounded set of finite fractional perimeter and let $F$ be a minimizer of $\mathcal{F}_h(\cdot, E)$. Then, for every $r \in (0,\gamma h^{1/1+s})$ and for every $x_0 \in \bd^* F,$ it holds
    \begin{equation}
    \label{3.11}
        \min\{|B_r(x_0)\setminus F|, |F \cap B_r(x_0)| \} \ge c r^N
    \end{equation}
    \begin{equation}\label{284}
        c r^{N-s} \le P^s(F,B_r(x_0)) \le C r^{N-s},
    \end{equation}
    where the constants only depend on $N$ and $s$.
\end{corollary}

\begin{proof}
Since $F$ is a minimizer of $\mathcal{F}_h(\cdot, E)$, for any $x_0 \in \bd F$, it holds that $\mathcal{F}_h(F,E) \le \mathcal{F}_h(F \cup B_r(x_0),E)$, which implies
\begin{align*}
    P^s(F) &\le P^s(F \cup B_r(x_0))+\frac{1}{h}\int_{B_r(x_0)\setminus F} \sd_E \ud x +\frac{1}{h^{s/1+s}} |B_r(x_0) \setminus F|\nonumber\\
    &\le P^s(F \cup B_r(x_0)) +\frac{C}{h^{s/1+s}}|B_r(x_0) \setminus F|.
\end{align*} 
Analogously, one can show that 
\begin{align}\label{292}
    P^s(F) &\le P^s(F\setminus B_r(x_0)) + \frac{C}{h^{s/1+s}}|F \cap B_r(x_0)|\\
    &= \Ls(F\setminus B_r(x_0),F^c\setminus B_r(x_0))+ \Ls(F\setminus B_r(x_0),B_r(x_0)) +
    \frac{C}{h^{s/1+s}}|F \cap B_r(x_0)|\nonumber
\end{align}
Therefore, by Proposition \ref{density_prop}, we deduce 
\begin{equation*}
    \min\left \{|F\cap B_r(x_0)|,|B_r(x_0)\setminus F|\right \} \ge c r^N \qquad \forall 0<r< \gamma h^{1/1+s}.
\end{equation*}
The first inequality in \eqref{284} is now an immediate consequence of the relative isoperimetric inequality. 
To prove the second inequality, by \eqref{292} we get
\begin{align*}
    P^s(F,B_r(x_0))&= \Ls(F\cap B_r(x_0), F^c)+ \Ls(F \setminus B_r(x_0), F^c \cap B_r(x_0)) \\
    &=P^s(F)-\Ls(F\setminus B_r(x_0), F^c\setminus B_r(x_0))\\
    &\le \Ls(F\setminus B_r(x_0), B_r(x_0)) + \frac{C}{h^{s/1+s}}|B_r(x_0) \setminus F|\\
    &\le P^s(B_r(x_0))+ \frac{C\gamma^s}{ r^s} \omega_N r^N \le C(N,s) r^{N-s},
\end{align*}
where we used that $r \le \gamma h^{1/1+s}$.
\end{proof}

We employ the density estimates above to bound the distance function between two consecutive sets of the discrete flow.

\begin{proposition}\label{prop L inf estimate}
    There exists a constant $\gamma=\gamma(N,s)>0$ with the following property. Let $F \subset \R^N$ be a bounded set of finite fractional perimeter and let $E$ be a minimizer of $\mathcal{F}_h(\cdot, F)$, then
    \begin{equation*}
        \sup_{E \triangle F} \dist_{\bd F}\le \gamma h^{1/1+s}.
    \end{equation*}
    \end{proposition}
\begin{proof}
Let $\gamma = \max\{3, 2^{s+1/s}P^s(B)^{1/s}/C^{1/s} \}$, where $C=C(N,s)$ is the constant given by the Proposition \ref{density_prop}. Let $c > \gamma$ and $x_0 \in E \triangle F$. Suppose by contradiction that $\dist_{\bd F}(x_0)> c h^{1/1+s}$. Since the other case is analogous, we assume $x_0 \in E \setminus F$. We then have
    \begin{equation}\label{sd_mag}
        \sd_F(x_0) >c h^{1/1+s}
    \end{equation} 
    and thus any ball $B_r(x_0)$ of radius $r \le c h^{1/1+s}/2$ is contained in $F^c$. By the minimality of $E$, we have $\mathcal{F}_h(E,F) \le \mathcal{F}_h(E \setminus B_r(x_0),F)$, therefore
    \begin{equation*}
        P^s(E) \le P^s(E \setminus B_r(x_0)) - \frac{1}{h}\int_{E \cap B_r(x_0)} \sd_F \ud x + \frac{1}{h^{s/1+s}} |E \cap B_r(x_0)|.
    \end{equation*}
    We use \eqref{sd_mag} and $r \le c h^{1/1+s}/2$ to infer that
    \[-\frac{1}{h}\int_{E \cap B_r(x_0)} \sd_F \ud x<-\frac{c }{2h^{s/1+s}}|E \cap B_r(x_0)|. \]
    Then we have
     \begin{equation}\label{3.8}
        P^s(E) \le P^s(E \setminus B_r(x_0))  -\frac{1}{h^{s/1+s}}\left (\frac{c }{2}- 1\right ) |E \cap B_r(x_0)|.
    \end{equation}
    By assumption $c>3$ and we can apply the Proposition \ref{density_prop} with $\mu =0$ and obtain
    \begin{equation}\label{3.9}
        Cr^N \le |E \cap B_r(x_0)| \qquad \forall 0 <r <\frac c2 h^{1/1+s}.
    \end{equation} 
    On the other hand, from \eqref{3.8} we deduce, for every $0<r< c h^{1/1+s}/2$, that
    \begin{equation}\label{3.10}
        \frac{1}{h^{s/1+s}}\left ( \frac{c}{2} -1\right )|E \cap B_r(x_0)|\le P^s(E\setminus B_r(x_0)) -P^s(E) \le P^s(B_r^c)=P^s(B) r^{N-s} 
    \end{equation}
    (where the last inequality follows from the subadditivity of the perimeter on $E$ and $B_r^c$). Combining \eqref{3.9} and \eqref{3.10}, we get that
    \[C r^N \le |E\cap B_r(x_0)| \le P^s(B) \left(\frac c2 -1\right)^{-1}h^{s/1+s}r^{N-s}\le 2P^s(B) h^{s/1+s}r^{N-s} \]
    for all $0<r<ch^{1/1+s}/2$, which gives the desired contradiction to the choice of $c$ as soon as $r\to ch^{1/1+s}/2$.
\end{proof}

We now characterize the stationary sets $E$ for the discrete flow. We say that $E$ is a  \textit{stationary set}  for the discrete flow if it is a fixed set for the functional \eqref{pb_min_1}, that is, 
\[ E=E_n\h\quad \forall n\in\N. \]

\begin{remark}\label{rmk bound}
From the monotonicity of the energy $P_s(\cdot) + h^{-\frac s{1+s}}||\cdot|-m|$ along the discrete flow, one can observe that, for $h=h(m,M)$ small, $|E_n^{(h)}|\in (m/2,3m/2)$ for all $n\in\N$.
\end{remark}

In the following, we will always assume that either:
\begin{itemize}
    \item $N=2$;
    \item $N\le 7$ and $s\in(s_0,1)$, where $s_0$ is the constant of Proposition \ref{prop density estimates}, item \textit{ii)}.
\end{itemize}
This hypothesis is essential for the proof of the following result.

\begin{proposition}   \label{prop 3.1 new}
    Every stationary set $E$ for the discrete flow is a critical set of the $s-$perimeter, that is, a single ball.
\end{proposition} 
\begin{proof}
    It is an immediate consequence of the Euler-Lagrange equation \eqref{euler-lagrange equation}. Since $E$ is a stationary point for the discrete flow, it satisfies
    \begin{equation*}
        \int_{\bd E} H^s_{E} \ud \mathcal{H}^{N-1} =
        \lambda \int_{\bd E} X \cdot \nu_E\ud \mathcal{H}^{N-1}
    \end{equation*}
    for all $X \in C^1_c(\R^N,\R^N)$, i.e. $E$ is a critical point for the $s-$perimeter. Employing \cite[Theorem 1.1]{CFMMT} and \cite[Theorem 1.1]{CiFiMaNo}, we conclude that $E$ is a single ball having fractional mean curvature $H^s_E=\lambda.$
\end{proof}

Before proving the convergence of the flow up to translations, we recall the following lemma from \cite{MPS} that will be used in the proof of the next proposition.

\begin{lemma}
Let $\{E_n\h\}_{n \in \N}$ be a volume preserving discrete flow starting from $E_0$ and let $E_{k_n}\h$ be a subsequence such that $E_{k_n}\h+\tau_n\to F$ in $L^1$ for some set $F$ and a suitable sequence $\{\tau_n\}_{n \in\N}\subset\R^N.$ Then $\dist_{\bd E_{k_n-1}\h}(\cdot+\tau_n)\to \dist_{\bd F}$ uniformly. 
\label{lemma 3.5n}
\end{lemma}

The following result proves the convergence of the discrete flow to a union of disjointed balls, all having the same radius. Moreover, we prove that the flow eventually has fixed volume. A this point, we can not rule out that the flow is converging to different balls (each at infinite distance from the others) and that the translations introduced are different along different subsequences. We will provide a sharper result in the final theorem.

\begin{proposition}\label{prop up to }
Let $m$, $M>0$ and $E_0$ be an initial bounded set with $P^s(E_0)\le M$, $|E_0|=m$. Then, for $h=h(s,M,m)$ small enough and for any discrete flow $E_n\h$ starting from $E_0$ the following hold: 
\begin{itemize}
    \item[i)] for $n$ sufficiently large $|E_{n}\h|=m$;
     \item[ii)] there exists $$ P_\infty^s=\lim_{n\to\infty} P^s(E_n\h); $$
    \item[iii)]  $E_{n}\h$ is made of $K=(P_{\infty}^s/ \omega_N^s)^{\frac{N}{s}} (\omega_N/m)^{\frac{N}{s}-1}$
    distinct connected components $E_{n,i}\h$, and $E_{n,i}\h-\textnormal{bar}(E_{n,i}\h)$ converges in $C^k$, for every $k \in \N$, to the ball centered at the origin and having mass $ m/K$. 
\end{itemize}
\end{proposition}

\begin{proof}
Let $\{E_{k_n}\h \}_{n \in \N}$ be any given subsequence of $\{E_n\h \}_{n \in \N}$. By Proposition \ref{prop density estimates}, each set $E_{k_n}\h $ is made up of $l_n \le k_0$ connected components having diameter uniformly bounded by $d_0$. Therefore, there exist $l_n$ balls $\{ B_{d_0}(\xi_n^i)\}$, each containing a different component of $E_{k_n}\h$ and such that $E_{k_n}\h \subset \cup_{i=1}^{l_n} B_{d_0}(\xi_n^i)$. Up to subsequences, we can assume that $l_n=\tilde l$, and for all $1\le i<j\le \tilde l$ the following limits exist
\[\limsup_{n \to \infty} |\xi_n^i-\xi_n^j|\eqqcolon d^{i,j}\in [0,+\infty].\]
Now we define the following equivalence classes: we say that $i \equiv j$ if and only if $d^{i,j}<+\infty.$ Denote by $l\le \tilde l$ the number of such equivalence classes, let $j(i)$ be a representative for each class $i \in \{1,\ldots,l\}$, and set $\sigma_n^i\coloneqq \xi^{j(i)}$ for $i=1,\ldots , l.$ We have constructed a subsequence $E_{k_n}\h$ satisfying $E_{k_n}\h \subset \cup_{i=1}^l B_R(\sigma_n^i),$ where $R=d_0 +\max \{d^{i,j}: d^{i,j}<+\infty \}+1,$ and for all $i \not \equiv j$ it holds $|\sigma_n^i -\sigma_n^j| \to +\infty$ as $n \to +\infty$.

Now, fix $1 \le i \le l$, and set
\begin{equation*}
    F_n^i \coloneqq E_{k_n}\h -\sigma_n^i, \quad 
    \tilde F_n^i \coloneqq (E_{k_n}\h -\sigma_n^i) \cap B_R, \quad
    m_n^i \coloneqq |\tilde F_n^i|.
\end{equation*}
Up to a subsequence, we have $m_n^i \to m^i>0$. Moreover, by Lemma \ref{lemma 3.5n} and 
by the compactness of sets of equi-bounded fractional perimeters, there exist $\tilde F^i \Subset B_R$ such that, up to a subsequence,
\begin{equation}\label{3.5}
    \tilde F_n^i \to \tilde F^i \textnormal{ in } L^1, \quad \sd_{E_{k_n -1}\h}(\cdot +\sigma_n^i) \to \sd_{\tilde F^i}(\cdot) \textnormal{ locally uniformly.}
\end{equation}
Let $\tilde G^i$ be any bounded set with $|\tilde G^i|=m_n^i$ and let $\tilde G_n^i \coloneqq \left ( \frac{m_n^i}{m^i} \right )^{\frac{1}{N}} \tilde G^i.$ We set now $G_n^i \coloneqq (F_n^i \setminus \tilde F_n^i) \cup \tilde G_n^i$ so that, for $n$ sufficiently large, $|F_n^i|=|G_n^i|$. By the minimality of $E_{k_n}\h$ we have
\begin{equation*}
    P^s(F_n^i) + \frac{1}{h} \int_{F_n^i} \sd_{E_{k_n -1}\h}(x+\sigma_n^i) \ud x \le P^s(G_n^i) + \frac{1}{h} \int_{G_n^i} \sd_{E_{k_n -1}\h}(x+\sigma_n^i) \ud x.
\end{equation*}
For $n$ sufficiently large, we obtain
\begin{align*}
    &P^s(\tilde F_n^i) + \int_{\tilde F_n^i} \int_{F_n^i \setminus \tilde F_n^i} \frac{1}{|x-y|^{N+s}} \ud x \ud y+ \frac{1}{h} \int_{\tilde F_n^i} \sd_{E_{k_n -1}\h}(x+\sigma_n^i) \ud x \\
    &\le P^s(\tilde G_n^i) + \int_{\tilde G_n^i} \int_{F_n^i \setminus \tilde F_n^i} \frac{1}{|x-y|^{N+s}} \ud x \ud y +\frac{1}{h} \int_{\tilde G_n^i} \sd_{E_{k_n -1}\h}(x+\sigma_n^i) \ud x.
\end{align*}
Passing to the limit as $n \to \infty$, using \eqref{3.5} and the uniform boundedness of $\tilde F_n^i$ and $\tilde G_n^i$, we deduce that
\begin{equation*}
      P^s(\tilde F^i) + \frac{1}{h} \int_{\tilde F^i} \sd_{\tilde F^i}(x) \ud x \le P^s(G^i) + \frac{1}{h} \int_{G^i} \sd_{\tilde F^i}(x) \ud x.
\end{equation*}
This minimality property extends by density to all competitors $G^i$ with finite perimeter and volume $m^i$, so that we deduce that $\tilde F^i$ is a fixed point for the discrete scheme with prescribed volume $m^i$, and, whence by Proposition \ref{prop 3.1 new}, it is a ball. Moreover, since $\tilde F^i$ are uniform $\Lambda-$minimizer by Proposition \ref{prop density estimates}, we also deduce that $\tilde F_n^i$ converge to $\tilde F^i$ in $C^{1,\alpha}$ for every $\alpha \in (0,1).$ In particular, for $n$ large enough, $\tilde F_n^i$ has only one connected component.

We have shown that, for $n$ large enough, $E_{k_n}\h$ is made up by a fixed number $K$ of connected components $E_{k_n}^{(h),i}$, $i=1, \ldots, K$ and $E_{k_n}^{(h),i} -\textnormal{bar}(E_{k_n}^{(h),i}) \to B_{R_i}$ where $|B_{R_i}|=m_i$.
Now, we show that all the radii $R_i$ are equal to $R$. To this aim, we consider the Euler-Lagrange equation \eqref{euler-lagrange equation} 
\begin{equation*}
   \frac{1}{h}  \sd_{E_{k_n-1}\h} + H^s_{ E_{k_n}\h} = \lambda_n \quad \text{on }   \bd E_{k_n}\h .
\end{equation*}
By Proposition \ref{prop L inf estimate}, we deduce that
\begin{equation*}
    |\lambda_n| \le h^{-1} \|\sd_{E_{k_n-1}\h}\|_{L^{\infty}(\bd E_{k_n}\h)} + \|H^s_{ E_{k_n}\h}\|_{L^{\infty}(\bd E_{k_n}\h)}\le c+  \|H^s_{ E_{k_n}\h}\|_{L^{\infty}(\bd E_{k_n}\h)}.
\end{equation*}
To bound the right hand side, we use the $\Lambda-$minimality of $E_{k_n}\h$ to obtain 
\[ \|H^s_{ E_{k_n}\h}\|_{L^{\infty}(\bd E_{k_n}\h)} \le \Lambda.\]
Therefore, by passing to a further subsequence, we can assume $\lambda_n \to \lambda \in \R.$ Arguing as before, we can localize the Euler-Lagrange equation to each single $F_n^i$ and obtain
\begin{equation*}
   \frac{1}{h}  \sd_{E_{k_n-1}\h}(x+\sigma_n^i) +  H^s_{ F_n^i}(x)  = \lambda_n \quad x\in \bd F_n^i.
\end{equation*}
We can pass to the limit as $n \to \infty$ thanks to Lemma \ref{lemma 3.5n} and the continuity property of the fractional mean curvature (see e.g. \cite[Lemma 2.1]{CiFiMaNo}). Thus, taking into account that $\tilde F^i$ is a fixed set for \eqref{pb_min_1}, we deduce that
\begin{equation*}
      H^s_{\tilde F^i}  = \lambda \quad \text{on }\bd \tilde F^i.
\end{equation*}
In particular, this shows that $R_i=c\lambda^{-s},$ for a suitable constant $c$ depending only on $s$ and $N$. 

In order to prove that, eventually, $|E_n\h|=m$, we proceed as follows. Set  $|B_{R_i}|= c_1 \lambda^{-sN}$ and $P^s(B_{R_i})= c_2 \lambda^{-s(N-s)}.$ From Remark \ref{rmk bound}, we take $h=h(s,M)$ small enough such that
\begin{equation*}
    |E_{k_n}\h| \in \left [\frac m2,\frac {3m}2 \right], \quad
    P^s(E_{k_n}\h) \le P^s(E_0)\le M
\end{equation*}
and, for $n$ large enough, this implies
\begin{equation*}
    \sum_{i=1}^K m_n^i \in \left [\frac m2,\frac {3m}2 \right], \quad
    \sum_{i=1}^K P^s(\tilde F_n^i) \le M.
\end{equation*}
Passing to the limit as $n \to \infty$ we obtain
$$K c_1 \lambda^{-sN} \in \left [\frac m2,\frac {3m}2 \right], \quad K c_2 \lambda^{-s(N-s)}\le M,$$
which implies 
\begin{equation}\label{contra}
    \lambda^{s^2}\le \dfrac{2c_1\,M}{m\, c_2}.
\end{equation}
If we suppose that $|E_{k_n}\h|\ne m$ for infinitely many indexes, then $\lambda=\textnormal{sgn}(m-|E_{k_n}\h|)h^{-\frac{s}{1+s}}$ which is a contradiction to \eqref{contra} if $h$ is sufficiently small. We have thus proved item \textit{i)}. Since, for $n$ large enough, $|E_n\h|=m$, the sequence $\{ P^s(E_n\h) \}_{n\in\N}$ is eventually non-increasing, from which item \textit{ii)} follows. Knowing the exact values of the volume and $s-$perimeter of any limit point, we are able to compute $K$ and obtain the convergence in $L^1$ of the whole sequence. Moreover, arguing as in \cite{CeNo} we conclude the convergence in $C^k$ for every $k \in \N$ via a bootstrap method.
\end{proof}

In order to prove the main theorem, we need to recall some results of \cite{MPS}.

\begin{lemma}
Let $\eta>0$. There exists $\delta>0$ with the following property: if $f_1, f_2\in C^1(\bd B)$ with $\|f_i\|_{C^1(\bd B)}\le \delta$ and $|B_{f_i}|=|B|$ for $i=1,2$ we have
\begin{align*}
    C_1(1-\eta)\|f_1-f_2\|_{L^2(B)}^2\le &\mathcal D(B_{f_1},B_{f_2})\le C_1 (1+\eta)\|f_1-f_2\|_{L^2(B)}^2\\[2ex]
    \dfrac{1-\eta }2\int_{\bd B_{f-1}}\sd^2_{ B_{f_2}}\udH\le &\mathcal D(B_{f_1},B_{f_2})\le \dfrac{1+\eta }2\int_{\bd B_{f-1}}\sd^2_{ B_{f_2}}\udH\\[2ex]
    |\textnormal{bar}(B_{f_1})-\textnormal{bar}(B_{f_2})|^2\le& C_2\|f_1-f_2\|^2_{L^2(B)} \le \frac{C_2}{C_1(1-\eta)}\mathcal D(B_{f_1},B_{f_2})
\end{align*}
for suitable constants $C_1,\, C_2>0$.
\end{lemma}

The following lemma  proves the crucial dissipation-dissipation inequality \eqref{dissipation-dissipation inequality}. This result will play a central role in the proof of Theorem \ref{main result}. Its proof is based on the Alexandrov-type estimate contained in Theorem \ref{Alex} and will be omitted as it is the same of \cite[Lemma 3.8]{MPS}.

\begin{lemma}
Let $h>0$. There exist constants $C(h,m,s),$ $\delta>0$  with the following property: given two normal deformations $B^{(m)}_{f_1},$ $B^{(m)}_{f_2}$ with $f_i\in C^2(\bd B^{(m)}),$ $\|f_i\|_{C^1(\bd B^{(m)})}\le \delta$, and such that $|B^{(m)}_{f_2}|=m,$ $\textnormal{bar}( B^{(m)}_{f_2})=0$ and
\begin{equation}
\label{dissipation-dissipation inequality}
    H_{B^{(m)}_{f_2}}+\dfrac {\sd_{B^{(m)}_{f_1}}}h=\lambda \quad \text{on}\quad \bd B^{(m)}_{f_2}
\end{equation}
for some $\lambda\in\R$, we have
\begin{equation*}
    \mathcal D(B^{(m)},B^{(m)}_{f_2})\le C\mathcal D(B^{(m)}_{f_2},B^{(m)}_{f_1}).
\end{equation*}
\end{lemma}

We are finally able to prove Theorem \ref{main result}. We will follow closely the proofs of \cite[Theorem 3.3]{MPS} and \cite[Theorem 1.1]{DeKu}. The main difference is that we use the fractional perimeter framework previously studied instead of the classical one. We present a sketch of the proof.

\begin{proof}
We start by proving the exponential decay of the dissipations, corresponding to Step 1 in \cite[Theorem 3.3]{MPS}. 

From Proposition \ref{prop up to } we know that any limit point of the discrete flow is given by the union of $K$ disjoint balls, all having volume $m/K$. We then use two competitors to obtain a discrete Gronwall-type inequality. Firstly, testing the minimality of $E_n\h$ with $E_{n-1}\h$ and summing from $n$ to infinity, we obtain 
\begin{equation*}
    \sum_{k\ge n-1} \mathcal D(E_k\h, E_{n-1}\h)\le  P(E_n\h)-P_\infty^s = P(E_n\h)-KP^s(B^{(m/K)}).
\end{equation*}
On the other hand, recalling Proposition \ref{prop up to }, the sets $(E_n\h)^i-\text{bar}((E_n\h)^i)=:(E_n\h)^i-\xi_n^i$ are eventually $C^{1,\alpha}-$deformations of $B^{(m/K)}$, having volume $|(E_n\h)^i|=m_n^i$. We consider the admissible competitor for $E_n\h$ given by
\[ \mathcal B_n = \bigcup_{i=1}^K \left( B^{(m_{n-1}^i)}+\xi_{n-1}^i \right). \]
Testing the minimality of $E_n\h$ against $\mathcal B_n$, one can obtain 
\[ P^s(E_n\h)-P^s(\mathcal B_n)\le C\mathcal D(E_{n-1}\h, E_{n-2}\h). \]
Recalling that, if a measurable set $F$ has $L$ disjointed connected components $F^i$, $i=1,\dots, L$, then 
\begin{equation*}
    P^s(F)= \sum_{i=1}^L P^s(F^i) - 2\sum_{i < j} \int_{F^i}\int_{F^j} \frac{1}{|x-y|^{N+s}} \ud x \ud y,
\end{equation*}
by concavity, we estimate
\[ P^s(\mathcal B_n)\le \sum_{i=1}^K P^s(B^{(m^i_{n-1})})\le  KP^s(B^{(m/K)}).  \]
Thus, combining the previous two estimates, we obtain the discrete Gronwall-type estimate
\[  \sum_{k\ge n-1} \mathcal D(E_k\h, E_{n-1} \h)\le C \mathcal D(E_{n-1}\h, E_{n-2}\h). \]
Finally, employing \cite[Lemma 3.10]{MPS} we conclude the exponential convergence of the dissipations
\[ \mathcal D(E_n\h,E_{n-1}\h)\le \left( 1-\frac 1{C+1}  \right)^{\frac n2} (P^s(E_0)- KP^s(B^{(m/K)}) ). \]
From now on, one can follow directly the proof of \cite[Theorem 3.3]{MPS} to conclude that the discrete flow $E_n\h$ is eventually contained in a compact set and converges in $C^k$ to a union of $K$ disjoint balls. Now, from Proposition \ref{prop 3.1 new} we deduce that the limit point is indeed a single ball, having volume equal to $m$, thus reaching the conclusion of the proof. 
\end{proof}

As a corollary of the previous result, we prove the following characterization of the long-time behaviour of the discrete volume-preserving mean curvature flow associated with the classical perimeter. The definition is analogous to the discrete flow associated with the fractional perimeter, but in this case we consider the following minimum problem (cp. the problem \eqref{pb_min_1})
\begin{equation*}
    \min \left \{ P(F)+\frac{1}{h} \int_{F} \sd_E(x) \ud x +\frac{1}{\sqrt h} ||F|-m| : F \subset \R^N\right \}.
\end{equation*}

A careful study of the previous results highlights that  the classical perimeter behaves essentially in the same manner. In particular, Proposition \ref{prop up to } still holds for $s=1$, replacing  \virg{$P^{1}$} with the classical perimeter. Thus, we conclude the following.

\begin{proposition}
    Let $m$, $M>0$ and let $E_0$ be an initial set with $P(E_0)\le M$, $|E_0|=m$. Then, for $h=h(M,m)$ small enough the following holds: for any discrete flow $E_n\h$ starting from $E_0$, there exist $x_i\in\R^N$, $i=1,\dots, K$, where $K=N^{-N}\omega_N m^{1-N} P_\infty^N$ such that 
    \[ E_n\h\to \bigcup_{i=1}^K \left( B^{(m/K)}+x_i\right)\qquad \text{in } C^k \]
    for all $k\in\N.$ Moreover, the convergence is exponentially fast.
\end{proposition}

The proof is essentially the same, the major difference being that the limit point of the flow now are union of disjoint balls (see \cite[Proposition 3.1]{MPS} for details) and not a single ball.

\appendix
\section{Existence of flat flows}
In order to prove the existence of flat flows, we follow the lines of \cite[Section 3]{MSS}. In this appendix we drop the assumption on the dimension and work in $\R^N$, $N\ge 2$. We start by remarking that the minimality of the sets $E_n\h$ implies
\begin{align}
\label{3.15}
   & P^s(E_t^{(h)}) +\frac{1}{h} \int_{E_t^{(h)} \triangle E_{t-h}^{(h)}} \dist_{\bd E_{t-h}^{(h)}}(x) \ud x + \frac{1}{h^{s/1+s}} ||E_{t}^{(h)}|-m| \nonumber\\
    &\le P^s(E_{t-h}^{(h)}) +  \frac{1}{h^{s/1+s}} ||E_{t-h}^{(h)}|-m| \qquad \forall t \in [h,+\infty)
\end{align}
and, by iterating \eqref{3.15} and using $|E_0|=m$, we get
\begin{align}
    P^s(E_t^{(h)}) \le P^s(E_0), \label{dis_per}\\[1ex]
    \frac{1}{h^{s/1+s}} ||E_{t}^{(h)}|-m| \le P^s(E_0),  \label{dis_pen} \\[1ex]
     \int_h^T \int_{E_{t}^{(h)} \triangle E_{t-h}^{(h)}} \frac{\dist_{\bd E_{t-h}^{(h)}}(x)}{h} \ud x \le P^s(E_0), \nonumber
\end{align}
for all $t\ge 0$ and for every $T>h$. We are then able to bound the $L^1-$distance between two consecutive sets of the discrete flow.
\begin{proposition}
\label{3.2.3}
Let $t >0$. Then
    \begin{equation*}
        |E_t^{(h)} \triangle E_{t-h}^{(h)}| \le C \left ( l^s P^s(E_{t-h}^{(h)})+ \frac{1}{l} \int_{E_{t}^{(h)} \triangle E_{t-h}^{(h)}} \dist_{\bd E_{t-h}^{(h)}}(x) \ud x  \right ) \quad \forall l \le \gamma h^{1/1+s}.
    \end{equation*}
\end{proposition}

\begin{proof}
In order to estimate $|E_t^{(h)} \triangle E_{t-h}^{(h)}|$, we split it into two parts:
\begin{equation*}
    |E_t^{(h)} \triangle E_{t-h}^{(h)}| \le |\{ x \in E_t^{(h)} \triangle E_{t-h}^{(h)}: \dist_{\bd F}(x) \le l\}|+ |\{x \in E_t^{(h)} \triangle E_{t-h}^{(h)}: \dist_{\bd E_{t-h}^{(h)}}(x) \ge l\}|.
\end{equation*}
    The second term is estimated by
\[ |\{ x \in E_t^{(h)} \triangle E_{t-h}^{(h)}: \dist_{\bd E_{t-h}^{(h)}}(x) \ge l\}| \le \frac{1}{l} \int_{E_t^{(h)} \triangle E_{t-h}^{(h)}} \dist_{\bd E_{t-h}^{(h)}}(x) \ud x.\]
To estimate the first term, we use a covering argument to find a collection of disjoint balls $\{ B_l(x_i)\}_{i \in I}$ with $x_i \in \bd E_{t-h}^{(h)}$ and $I \subset \N$ a finite set such that $\bd E_{t-h}^{(h)} \subset \cup_{i \in I} B_{2l}(x_i)$. Observe that, by \eqref{3.11}, \eqref{284} and the relative isoperimetric inequality, for every $i \in I$, we get
\begin{align*}
    |B_{3l}(x_i)| &\le C \min\{|E_{t-h}^{(h)} \cap B_l(x_i)|, |B_l(x_i) \setminus E_{t-h}^{(h)}| \} \\
    &\le C \Ls(E_{t-h}^{(h)}\cap B_l(x_i), B_l(x_i) \setminus E_{t-h}^{(h)} )^{N/N-s} \\
    &\le C l^s\,\Ls(E_{t-h}^{(h)} \cap B_l(x_i), B_l(x_i) \setminus E_{t-h}^{(h)})\\
    &\le C l^s\,\Ls(E_{t-h}^{(h)} \cap B_l(x_i), \R^N \setminus E_{t-h}^{(h)}).
\end{align*}
Since the set $\{x \in E_t^{(h)} \triangle E_{t-h}^{(h)}: \dist_{\bd E_{t-h}^{(h)}}(x) \le l\}$ is covered by $\{B_{3l}(x_i)\}_{i \in I}$, by summing over $i$ and by the choice of the balls $\{ B_l(x_i)\}_{i \in I}$, we obtain
\begin{align*}
    |\{ x \in E_t^{(h)} \triangle E_{t-h}^{(h)}: \dist_{\bd E_{t-h}^{(h)}}(x) \le l\}| &\le \sum_{i\in I} |B_{3l}(x_i)|\\
    &\le C l^s \sum_{i \in I} \Ls(E_{t-h}^{(h)} \cap B_l(x_i), \R^N \setminus E_{t-h}^{(h)})\\
    &\le  C l^s P^s(E_{t-h}^{(h)}).
\end{align*}
\end{proof}

\begin{proposition}
Let $h \le 1$, then it holds
\begin{equation}\label{holder estimate}
    |E_{t_1}^{(h)} \triangle E_{t_2}^{(h)}| \le C |t_1-t_2|^{s/s+1} \quad \forall 0\le t_1 \le t_2 <+\infty.
\end{equation}
\end{proposition}
\begin{proof}
It is enough to consider the case $t_2-t_1 \ge h$. Let $j \in \N$ and $k \in \N \setminus\{0\}$ be such that $t_1 \in [j h,(j+1)h)$ and $t_2 \in [(j+k) h, (j+k+1)h).$ Then, we can use Proposition~\ref{3.2.3} with $l=\gamma h/|t_1-t_2|^{s/s+1}$ (note that $l \le \gamma h^{1/s+1}$ by the assumption $t_2-t_1 \ge h$) and estimate in the following way
\begin{align*}
    |E_{t_1}\h \triangle E_{t_2}\h| &\le \sum_{i=1}^k |E_{(j+i)h}\h \triangle E_{(j+i-1)h}\h|\\
    &\le C \sum_{i=1}^k \frac{h^s}{|t_1-t_2|^{s^2/s+1}} P^s(E_{(j+i-1)h}\h)\\
    & \ +C \sum_{i=1}^k \frac{|t_1-t_2|^{s/s+1}}{h} \int_{E_{(j+i)h}\h \triangle E_{(j+i-1)h}}\h |\sd_{E_{(j+i-1)h}\h}| \ud x.
\end{align*}
By using \eqref{3.15} we estimate the sum above by 
\begin{align*}
    |E_{t_1}\h \triangle E_{t_2}\h| &\le C \sum_{i=1}^k \frac{h^s}{|t_1-t_2|^{s^2/s+1}} P(E_0)\\
    & \ + C \sum_{i=1}^k |t_1-t_2|^{s/s+1} (P(E_{(j+i-1)h}\h)-P(E_{(j+i)h}\h))\\
    & \ + C \sum_{i=1}^k \frac{|t_1-t_2|^{s/s+1}}{h^{s/s+1}} \left ( ||E_{(j+i-1)h}\h|-m| -||E_{(j+i)h}\h|-m| \right )\\
    &\le C \frac{k h^s}{|t_1-t_2|^{s^2/s+1}} P(E_0)+ C |t_1-t_2|^{s/s+1}(P(E_{t_1}\h)-P(E_{t_2}\h))\\
    & \ + C \frac{|t_1-t_2|^{s/s+1}}{h^{s/s+1}} \left ( ||E_{t_1}\h|-m| -||E_{t_2}\h|-m| \right ).
\end{align*}
Now, by \eqref{dis_per} and \eqref{dis_pen}, we get
\begin{equation*}
    |E_{t_1}\h \triangle E_{t_2}\h| \le C |t_1-t_2|^{s/s+1} P(E_0),
\end{equation*}
where we used that $k h^s \le 2 |t_1-t_2| h^{s-1} \le 2|t_1-t_2|^s$ since $h \le |t_1-t_2|$.
\end{proof}

We are now able to prove the existence of fractional flat flows, defined as $L^1_{loc}-$limit points of the discrete flow previously introduced.  
\begin{proposition}\label{flat flows}
    Let $E_0$ be a bounded initial set of finite fractional perimeter and volume $m$. For any $h>0$, let $\{ E_t\h \}_{t\ge 0}$ be a discrete flow. Then, there exist a family of sets $\{  E_t\}_{t\ge 0}$ of finite fractional perimeter  and a subsequence $h_k\to 0$ such that, as $k\to \infty$, it holds
    $$  E^{(h_k)}_t\to E_t \quad \text{in } L^1_{loc}, \quad \text{for all }t\in [0,+\infty).  $$
    Moreover, the flow satisfies $\forall\, 0\le s\le t$,
    \begin{align*}
        |E_t\triangle E_s|&\le C |t-s|^{\frac s{1+s}},\\
        P^s(E_t)&\le P^s(E_0).
    \end{align*}
\end{proposition}

\begin{proof}
We follow \cite[Theorem 2.2]{MSS}. Considering $t\in\Q^+$, from \eqref{dis_per} and the compactness of sets of finite fractional perimeter, we find a subsequence $h_k\to 0$ such that 
\[ L^1_{loc}-\lim_{k\to \infty}E_t^{(h_k)}=E_t\quad \forall t\in \Q^+. \]
By the triangular inequality, it's easy to see that \eqref{holder estimate} passes to the limit. Finally, a simple continuity argument implies that the whole sequence $E^{(h_k)}_t$ converges to the sets $E_t$ for all $t\in [0,+\infty)$.
\end{proof}

If we assume the following hypothesis:
\begin{enumerate}
    \item \label{HP}
    given $T>0$ and an initial bounded set $E_0$ of finite fractional perimeter, there exists $R>0$ independent of $h$ such that $E_t\h\subset B_R$ for all $h>0, t\in [0,T]$,
\end{enumerate}
we are able to prove that the flat flow is indeed volume-preserving.
\begin{corollary}
    
    Under hypothesis \ref{HP}, let $E_0$ be a bounded initial set of finite fractional perimeter and volume $m$. For any $h>0$, let $\{ E_t\h \}_{t\ge 0}$ be a discrete flow. Then, there exist a family of sets $\{  E_t\}_{t\ge 0}$ of finite fractional perimeter  and a subsequence $h_k\to 0$ such that, as $k\to \infty$, it holds
    $$  E^{(h_k)}_t\to E_t \quad \text{in } L^1 \quad \text{for all }t\in [0,+\infty).  $$
    Moreover, the flow satisfies $\forall\, 0\le s\le t$,
    \begin{align*}
        |E_t|&=|E_0|\\
        |E_t\triangle E_s|&\le C |t-s|^{\frac s{1+s}},\\
        P^s(E_t)&\le P^s(E_0).
    \end{align*}
\end{corollary}

\begin{proof}
    The proof is analogous to the one of Proposition \ref{flat flows}. By uniform boundedness, the limits are in $L^1$ instead of $L^1_{loc}$. Then, passing to the limit $h\to 0$ in \eqref{3.15} we conclude that $|E_t|=m$ for all $t\ge0$.
\end{proof}


\printbibliography 
\end{document}